\documentclass[11pt]{article}
\usepackage{amsmath}
\usepackage{amsthm}
\usepackage{amssymb}
\usepackage[dvipdfm]{hyperref}

\newcommand{\C}{\mathbb{C}}
\newcommand{\R}{\mathbb{R}}

\theoremstyle{plain}
\newtheorem{theorem}{Theorem}[section]
\newtheorem{proposition}[theorem]{Proposition}
\newtheorem{lemma}[theorem]{Lemma}
\newtheorem{Corollary}[theorem]{Corollary}

\theoremstyle{definition}
\newtheorem{definition}[theorem]{Definition}

\makeatletter
\@addtoreset{equation}{section}
\makeatother

\begin{document}


\title{A perturbation and generic smoothness of the Vafa--Witten moduli spaces 
on closed symplectic four-manifolds  
} 
\author{  
Yuuji Tanaka}
\date{}


\maketitle

\begin{abstract}
We prove a Freed--Uhlenbeck style generic smoothness
 theorem for the moduli space of solutions to the Vafa--Witten equations  
on a closed symplectic
 four-manifold by using a method developed by Feehan 
for the study of the
 $PU(2)$-monopole equations on smooth closed four-manifolds. 
We introduce a set of perturbation terms to the Vafa--Witten equations, 
 and prove that the moduli space of solutions to the perturbed
 Vafa--Witten equations  
on a closed symplectic four-manifold for the structure group $SU(2)$ or $SO(3)$
is a smooth manifold of dimension zero for a generic choice of
 the perturbation parameters. 
\end{abstract}





\section{Introduction}

In this article, 
we consider the Vafa--Witten equations (\cite{VW}, \cite{BM}, \cite{Hay}, \cite{W}) 
on a compact symplectic four-manifold. 
First let us introduce the equations in their original form.

\paragraph{The Vafa--Witten equations.}

Let $X$ be a closed, oriented, smooth Riemannian four-manifold with Riemannian
metric $g$, and let $P \to X$ be a principal $G$-bundle over $X$ 
with compact Lie group $G$. 
We denote by $\mathcal{A}_{P}$ the set of all connections of $P$ 
and by $\Omega^{+} (X, \mathfrak{g}_{P})$ the set of self-dual two-forms
valued in the adjoint bundle $\mathfrak{g}_{P}$ of $P$. 
We consider the following equations for a triple $(A, B , \Gamma) \in
\mathcal{A}_{P} 
\times \Omega^{+} (X , \mathfrak{g}_{P}) \times \Omega^{0} (X ,
\mathfrak{g}_{P})$: 
\begin{gather}
d_{A} \Gamma + d_{A}^{*} B = 0,  
\label{VW1} \\ 
F_{A}^{+}  + \frac{1}{8} [B.B ] + \frac{1}{2} [ B , \Gamma] =0 , 
\label{VW2}
\end{gather}
where $[B. B ] \in \Omega^{+} (X, \mathfrak{g}_{P})$ is 
defined by a point-wise Lie-algebraic structure on $\Lambda^{+}$
 together with the bracket of $\mathfrak{g}_{P}$ 
(see \cite[\S A.1]{BM} or \cite[\S 2]{tanaka1} for more detail). 
We call these equations the {\it Vafa--Witten equations}. 
The equations \eqref{VW1} and \eqref{VW2} with a gauge fixing equation
form an elliptic system with the index always being zero.

\paragraph{The equations on compact symplectic four-manifolds and a perturbation.}

We rewrite the equations \eqref{VW1} and \eqref{VW2} when the underlying
manifold $X$ is a compact symplectic four-manifold.

Let $X$ be a compact symplectic four-manifold with symplectic form
$\omega$. 
We take an almost complex structure $J$ compatible with 
the symplectic form $\omega$. 
In this setting, the equations \eqref{VW1} and \eqref{VW2} can be
written as follows (see
Section \ref{sec:pert} for more detail).  
\begin{gather*}
\bar{\partial}_{A} \alpha + \bar{\partial}_{A}^{*} \beta = 0 , 
\\
F_{A}^{0,2} + \frac{1}{2} [ \alpha , \beta ] = 0 , 
\, \, 
\omega^2 \wedge \left( i \Lambda F_{A}^{1,1} +  \frac{1}{2} [ \alpha ,
 \alpha^{*} ] 
 \right)+ [ \beta , \beta^{*} ]= 0 ,  
\end{gather*} 
where $\Lambda := (\wedge \omega)^{*}$, and $\alpha \in
\Omega^{0,0} (X , \mathfrak{g}_{P})$, $\beta \in \Omega^{0,2} (X, \mathfrak{g}_{P})$.

We then introduce the following perturbation for the Vafa--Witten
equations on a compact symplectic four-manifold: 
\begin{equation}
\bar{\partial}_{A} \alpha + \bar{\partial}_{A}^{*} \beta 
+ \rho (\theta) (\alpha + \beta ) = 0 , 
\label{eq:pVW1_i}
\end{equation}
\begin{equation}
F_{A}^{0,2} + \frac{1}{2} \tau_{1} [ \alpha , \beta ] = 0 , 
\, \, 
\omega^2  \wedge 
\left(i \Lambda F_{A}^{1,1}  + \frac{1}{2} \tau_{2} [ \alpha , \alpha^{*} ]  \right)+ \tau_{3} [ \beta , \beta^{*} ] = 0 , 
\label{eq:pVW2_i}
\end{equation}
where 
$\rho : T^{*} X \otimes \C \to \text{Hom}_{\C} (\Lambda^{0,0} \oplus
\Lambda^{0,2} , \Lambda^{0,1})$ is the Clifford multiplication map, 
$\tau_{1} \in C^{r} (GL (\Lambda^{0,2}))$, 
$\tau_{2} \in C^{r} (GL (\Lambda^{0,0}))$, 
$\tau_{3} \in C^{r} (GL (\Lambda^{2,2}))$ 
and $\theta \in T^{*} X \otimes \C$ 
are perturbation parameters.   
We write $\tau := ( \tau_1 , \tau_2 , \tau_3)$, 
and denote by $\mathcal{P}_1$ the Banach space of the perturbation parameters
$(\tau , \theta)$, 
namely, we set 
$\mathcal{P}_1 := 
C^r (GL (\Lambda^{0,2})) \times 
C^r (GL (\Lambda^{0,0})) \times 
C^r (GL (\Lambda^{2,2})) \times 
C^r ( \Lambda^{1} \otimes \C )$.

This perturbation 
does not depend upon connections. 
Hence, one needs not be careful about the compatibility with the
bubbling-off of connections.

\paragraph{Generic smoothness of the moduli spaces.}

Before stating results in this article, let us introduce some
terminology here first.

\begin{definition}
A connection $A$ of a principal $G$-bundle over $X$ is said to be 
{\it irreducible} if the stabilizer $Z_{A}$ in $\mathcal{G}_{P}$ coincides
with the centre of the group $G$, and {\it reducible} otherwise. 
\end{definition}

We also introduce the following notion of
{\it rank} for sections.

\begin{definition}
We say a $\mathfrak{g}_{P}$-valued form $\alpha + \beta \in \Gamma ( 
\mathfrak{g}_{P} \otimes ( \Lambda^{0,0} 
 \oplus \Lambda^{0,2} ) )$ is 
{\it of rank}  
 $r$ if, when considered as a section of $\text{Hom} ((\Lambda^{0,0} \oplus
 \Lambda^{0,2})^{*} , \mathfrak{g}_{P} )$, 
the section $( \alpha + \beta ) (x)$ has rank less than or equal to $r$ at every
point $x \in X$ with equality at some point. 
\end{definition}

We then denote by $\mathcal{M}^{*}_{\diamond} (\tau , \theta)$ the
moduli space of solutions $(A, (\alpha ,\beta))$ 
to the perturbed Vafa--Witten equations 
\eqref{eq:pVW1_i} and \eqref{eq:pVW2_i} with $A$ irreducible, 
$\alpha - \bar{\alpha} =0$ and $\alpha + \beta $
being of rank three. 
We prove the following in Section \ref{sec:rankthree}:

\begin{proposition}
Let $X$ be a closed symplectic four-manifold, and let $P \to X$ be a
 principal $G$-bundle over $X$, where we assume that $G$ is either $SU(2)$ or $SO(3)$. 
Then there is a first category subset $\mathcal{P}_1' \subset \mathcal{P}_1$ 
such that, for each $(\tau , \theta) \in \mathcal{P}_1 \setminus \mathcal{P}_1'$, 
the moduli space 
$\mathcal{M}^{*}_{\diamond} (\tau , \theta)$ is 
a smooth manifold of dimension zero.  
\label{th:rank3}
\end{proposition}

Here, a subset of $S'$ of a topological space $S$ is said to be 
a {\it first category subset} if 
$S'$ is a countable union of closed subsets of $S$ with empty
interior. 
We mean a {\it generic choice} of elements in $S$ 
by taking an element from $S \setminus S'$.

We next consider 
irreducible solutions to the equations with rank less than or equal to two, and show that there are no such solutions  
for a generic choice of perturbation parameters. 
In order to do this we further perturb the equations, that corresponds
to moving metrics or almost complex structures of the underlying manifold. 
More precisely, we introduce an extra perturbation parameter 
$f \in C^{r} (GL (T^*  X))$, and consider the following equations: 
\begin{equation}
\bar{\partial}_{A,f} \alpha + \bar{\partial}_{A,f}^{*} \beta 
+ \rho (f( \theta ) ) (\alpha + \beta ) = 0 , 
\label{eq:p2VW1_i}
\end{equation}
\begin{equation}
P^{0,2}_{f} ( F_{A} ) + \frac{1}{2} \tau_{1} [ \alpha , \beta ] = 0 , 
\, \, 
\omega^2 \wedge \left( i \Lambda P^{1,1}_{f} ( F_{A} )  
+ \frac{1}{2} \tau_{2} [ \alpha , \alpha^{*} ] \right)+ \tau_{3} [ \beta , \beta^{*} ] = 0 , 
\label{eq:p2VW2_i}
\end{equation}
where $P^{0,2}_{f}$ and $P^{1,1}_{f}$ are the projections to $(0,2)$ and
$(1,1)$-parts with respect to the almost complex structure $f^{*}J$, 
and 
\begin{equation*}
\bar{\partial}_{A,f} 
 := \sum f (v^{i} ) \wedge \nabla_{A,v_{i}},
 \quad 
\bar{\partial}_{A,f}^{*} 
 := - \sum \iota ( f ( v^{i} ) ) \nabla_{A, v_{i}}, 
\end{equation*}
where $\{ v^{i} \}$ is an orthonormal frame of $\Lambda^{0,1}$, 
and $\{ v_{i} \}$ is its dual. 
These $\bar{\partial}_{A,f}$ and $\bar{\partial}_{A,f}^{*}$ can be seen
as a variation of the Dirac operator corresponding to moving metrics or
almost complex structures of the
underlying manifold.

We denote by 
$\mathcal{P}_2 := 
C^r (GL (T^{*} X )) \times 
C^r ( \Lambda^{1} \otimes \C )
$ the perturbation parameter space and by $\mathcal{M}^{*, 0} (f, \theta)$
the moduli space of 
solutions $(A, (\alpha, \beta))$ to the equations \eqref{eq:p2VW1_i} and \eqref{eq:p2VW2_i}
with $A$ irreducible, $\alpha - \bar{\alpha} =0$ and $(\alpha , \beta) \neq 0$. 
We prove the following in Section \ref{sec:rankletwo}:

\begin{proposition}
Let $X$ be a closed symplectic four-manifold, and let $P \to X$ be a
 principal $G$-bundle over $X$, where the structure group $G$ is either $SU(2)$ or $SO(3)$. 
Then there is a first category subset $\mathcal{P}_2' \subset \mathcal{P}_2$ 
such that for all $(f ,  \theta) \in \mathcal{P}_2 \setminus \mathcal{P}_2'$, 
the moduli space $\mathcal{M}^{*, 0} (f,  \theta)$ 
contains no solutions $(A, (\alpha ,\beta))$ 
 to the perturbed Vafa--Witten equations \eqref{eq:p2VW1_i} and \eqref{eq:p2VW2_i} such that
$A$ is irreducible, $\alpha - \bar{\alpha} =0$ and $\alpha + \beta$ is of rank one or two. 
\label{th:norot}
\end{proposition}

Our proof of Proposition \ref{th:norot} invokes a series of ideas by
Feehan \cite{F} 
in the study
of the $PU(2)$-monopole equations, which uses a version of
the Sard--Smale theorem (see Section \ref{sec:term}).  
Note that Teleman \cite{Teleman}  independently obtained a similar generic-parameter smoothness result for the $PU(2)$-monopole moduli spaces on closed four-manifolds as well.

We now take   
$\mathcal{P} 
= C^r (GL (T^{*}X)) \times 
C^r (GL (\Lambda^{0,2})) \times 
C^r (GL (\Lambda^{0,0})) \times 
C^r (GL (\Lambda^{2,2})) \times 
C^r ( \Lambda^{1} \otimes \C )$ 
as the perturbation parameter space.  
Combining Propositions \ref{th:rank3} and \ref{th:norot} above, we obtain the following:

\begin{theorem}
Let $X$ be a closed symplectic four-manifold, and let $P \to X$ be a
 principal $G$-bundle over $X$, where the structure group $G$ is either $SU(2)$ or $SO(3)$. 
We denote by $\mathcal{M}^{* , 0} (f, \tau , \theta )$ the moduli space
of solutions $(A, (\alpha ,\beta))$ to the perturbed Vafa--Witten equations \eqref{eq:p2VW1_i} and
\eqref{eq:p2VW2_i} with $A$ irreducible, $\alpha - \bar{\alpha} =0$ and
$(\alpha , \beta ) \neq 0$.  
Then there is a first category subset $\mathcal{P}' \subset \mathcal{P}$ 
such that for all $( f, \tau , \theta) \in \mathcal{P} \setminus \mathcal{P}'$, 
the moduli space $\mathcal{M}^{*, 0} ( f, \tau , \theta)$ 
is a smooth manifold of dimension zero. 
\label{th:main}
\end{theorem}

Note that the $C^r$-perturbation parameter space
$\mathcal{P}$ and its first category subset in the above theorem can be replaced by
$C^{\infty}$-perturbation parameter space and its first category
subset by using an argument by Feehan--Leness $\cite[\S 5.1.2]{FL}$.

\vspace{0.1cm}

\paragraph{Acknowledgements.}
I would like to thank Kael Dixon and Alex Kinsella for helpful conversations to improve the presentation of this article.

\section{Perturbations}
\label{sec:pert}

We recall some descriptions of 
$Spin^{c}$-structures and the Dirac operators on compact symplectic
manifolds in Section \ref{sec:spin}. 
We then describe the perturbations to the equations on compact
symplectic four-manifolds in Sections \ref{sec:perteq} and \ref{sec:fpert}.

\subsection{Spinor bundles and the Dirac operator on symplectic manifolds}
\label{sec:spin}

A general reference for $Spin^{c}$-structures and the Dirac operators is
\cite{LM}.

\paragraph{Spinor bundles.}

A spinor bundle $S$ splits into the direct sum of vector
bundles $S^{+}$ and $S^{-}$, where $S^{+}, S^{-}$ are the eigenspaces of
the Clifford element of $\pm 1$ eigenvalues respectively. 
If $X$ is an oriented smooth four-manifold with $Spin^c$-structure, 
we have the following isomorphism induced from the Clifford multiplication: 
\begin{equation*}
T^{*} X \otimes \C \cong \text{\rm Hom}_{\C} (S^{+} ,
S^{-} ). 
\label{eq:isom}
\end{equation*}
See \cite{Mo} (or \cite[A.3]{F}) for a proof.  
If $X$ is an almost complex four-manifold, 
this isomorphism can be written as 
\begin{equation*}
T^{*} X \otimes \C \cong \text{\rm Hom}_{\C} (\Lambda^{0,0} \oplus 
\Lambda^{0,2} , \Lambda^{0,1}) . 
\label{eq:spin_isom}
\end{equation*}

\paragraph{The Dirac operator on symplectic manifolds.}

Let $E$ be a vector bundle on $X$. 
The Dirac operator $D_{A}$ associated to a connection $A$ on $E$ is
given by the  composition: 
\begin{equation*}
 \Gamma(S) \xrightarrow{\nabla_{A}} \Gamma(T^{*} X \otimes (S \otimes E)) 
\xrightarrow{metric} \Gamma(TX \otimes (S \otimes E)) \xrightarrow{\rho} \Gamma( S
\otimes E) , 
\end{equation*}
where $\rho$ is the Clifford multiplication map.

In the almost complex case, the Dirac
operator is  written as 
\begin{equation*}
 D_A = \sqrt{2} ( \bar{\partial}_{A} + \bar{\partial}_{A}^{*}),
\end{equation*}
where $A$ is a connection on $E$. 
Thus, if the underlying manifold $X$ is a symplectic four-manifold,
the Dirac equations become 
$ \bar{\partial}_{A} \alpha + \bar{\partial}_{A}^{*} \beta = 0$, 
where $\alpha \in
\Omega^{0, 0} (E)$, $\beta \in \Omega^{0,2} (E)$.

\subsection{The equations on symplectic four-manifolds and a perturbation}
\label{sec:perteq}

Let $X$ be a compact symplectic four-manifold with symplectic form
$\omega$, and let $P$ be a principal $G$-bundle over $X$, where $G$ is a
compact Lie group. 
We take an almost complex structure $J$ compatible with 
the symplectic form $\omega$.

Let us rewrite the equations \eqref{VW1} and \eqref{VW2}, 
when the underlying manifold is a compact symplectic four-manifold.  
This was thoroughly described by Mares \cite[\S 7]{BM}. 
We follow his notations. 
First we denote an orthonormal frame of $\Lambda^1$ by 
$\{ e^{0} , e^{1}, e^{2} , e^{3} \}$. 
We write
$d z^{1} = e^{0} + i e^{1}, \, d z^{2} = e^{2} + i e^{3}$. 
Note that we have $\omega = e^{0} \wedge e^{1}  + e^{2} \wedge e^{3}$. 
We write $B \in \Omega^{+} (\mathfrak{g}_{P})$ as 
$B = B_{1} (e^{0} \wedge e^{1} + e^{2} \wedge e^{3}) 
+ B_{2} (e^{0} \wedge e^{3} + e^{3} \wedge e^{1}) 
 +B_{3} ( e^{0} \wedge e^{3}+ e^{1} \wedge e^{2}) $. 
We then define $\alpha \in \Omega^{0,0} (X , \mathfrak{g}_{P})$ and 
$\beta \in \Omega^{0,2} (X , \mathfrak{g}_{P})$ by 
\begin{equation*}
\alpha := \Gamma + i B_{1} , \, \beta := - \frac{1}{2} (B_{2} + i B_{3}) 
 d \bar{z}^{1} \wedge d \bar{z}^{2} . 
\end{equation*} 
Note that $B$ can be written as $B = B_{1} \omega + \beta + \beta^{*}$. 
Note also that $\alpha - \bar{\alpha} =0$ if $A$ is irreducible, since $\Gamma =0$ in
this case.

With these notations, the equations \eqref{VW1} and \eqref{VW2} are
rewritten as follows. 
\begin{gather}
 \bar{\partial}_{A} \alpha + \bar{\partial}_{A}^{*} \beta = 0 , 
\label{VWS1} \\
F_{A}^{0,2} + \frac{1}{2} [ \alpha , \beta ] = 0 , 
\, \, 
\omega^2 \wedge \left( i \Lambda  F_{A}^{1,1}  + \frac{1}{2} [\alpha , \alpha^{*} ] 
 \right) + [ \beta , \beta^{*} ] = 0 ,   
\label{VWS2}
\end{gather} 
where $\Lambda := (\wedge \omega)^{*}$.

\paragraph{Perturbation.}

We consider the following perturbed Vafa--Witten equations: 
\begin{equation}
\bar{\partial}_{A} \alpha + \bar{\partial}_{A}^{*} \beta 
+ \rho (\theta) (\alpha + \beta ) = 0 , 
\label{eq:pVW1}
\end{equation}
\begin{equation}
F_{A}^{0,2} + \frac{1}{2} \tau_{1} [ \alpha , \beta ] = 0 , 
\, \, 
\omega^2 \wedge \left( i \Lambda F_{A}^{1,1} + \frac{1}{2} \tau_{2} [
		 \alpha , \alpha^{*} ] 
\right)+ \tau_{3} [ \beta , \beta^{*} ] = 0 , 
\label{eq:pVW2}
\end{equation}
where 
$\rho : T^{*} X \otimes \C \to \text{Hom}_{\C} (\Lambda^{0,0} \oplus
\Lambda^{0,2} , \Lambda^{0,1})$ is the Clifford multiplication,  
$\tau_{1} \in C^{r} (GL (\Lambda^{0,2}))$, 
$\tau_{2} \in C^{r} (GL (\Lambda^{0,0}))$, 
$\tau_{3} \in C^{r} (GL (\Lambda^{2,2}))$ 
and $\theta \in T^{*}X \otimes \C$ are perturbation parameters.

Note that this perturbation does not involve
connections. 
In Section \ref{sec:rankthree}, we prove that the moduli space of
solutions to the above equations \eqref{eq:pVW1} and \eqref{eq:pVW2} with
$A$ irreducible, $\alpha - \bar{\alpha} =0$ and $\alpha + \beta$ being of rank three is a smooth
manifold of dimension zero for a generic choice of the perturbation
parameters.

\subsection{Further perturbation}
\label{sec:fpert}

Following Feehan \cite[\S 3]{F}, we consider a perturbation of the
Dirac operator. 
We take $f \in C^{r} (GL (T^{*} X))$, 
and consider the following: 
\begin{equation*}
\bar{\partial}_{A,f} 
 := \sum f (v^{i} ) \wedge \nabla_{A,v_{i}},
 \quad 
\bar{\partial}_{A,f}^{*} 
 := - \sum \iota ( f ( v^{i} ) ) \nabla_{A, v_{i}}, 
\end{equation*}
where $\{ v^{i} \}$ is an orthonormal frame of $\Lambda^{0,1}$ and $\{
v_{i} \}$ is its dual. 
These $\bar{\partial}_{A,f}$ and $\bar{\partial}_{A,f}^{*}$ can be seen
as a variation of the Dirac operator corresponding to moving metrics or
almost complex structures of the
underlying manifold.

We then consider the following equations: 
\begin{equation}
\bar{\partial}_{A,f} \alpha + \bar{\partial}_{A,f}^{*} \beta 
+ \rho (f( \theta ) ) (\alpha + \beta ) = 0 , 
\label{eq:p2VW1}
\end{equation}
\begin{equation}
P^{0,2}_{f} ( F_{A} ) + \frac{1}{2} \tau_{1} [ \alpha , \beta ] = 0 , 
\, \, 
\omega^2 \wedge \left( i \Lambda P^{1,1}_{f} ( F_{A} )  
+ \frac{1}{2} \tau_{2} [ \alpha , \alpha^{*} ]  \right)+ \tau_{3} [ \beta , \beta^{*} ] = 0 , 
\label{eq:p2VW2}
\end{equation}
where $\theta \in T^{*}X \otimes \C $, 
$P^{0,2}_{f}$ and $P^{1,1}_{f}$ are the projections to $(0,2)$ and
$(1,1)$-parts with respect to the almost complex structure $f^{*}J$. 
We denote the left hand side of \eqref{eq:p2VW1} by 
$\left( \bar{\partial}_{A, (f,\theta)} +
\bar{\partial}^{*}_{A,(f,\theta)} \right) 
(\alpha + \beta)$.

As in \cite[Lem.~3.2]{F}, 
the differential of the above perturbed Dirac operator 
$\left( 
\bar{\partial}_{A,(f,\theta)} + \bar{\partial}^{*}_{A,(f,\theta)} \right)$ is given by 
\begin{equation*}
\begin{split}
D \left( \bar{\partial}_{A,(f,\theta)} +
 \bar{\partial}^{*}_{A,(f,\theta)} \right)_{(A,(f,\theta))} 
& ( a , \underline{f}, \underline{\theta}) 
(\mathfrak{a} + \mathfrak{b}) \\ 
&= \sum \underline{f} (v^{i} ) \wedge \nabla_{A , v_{i}}  \mathfrak{a} 
 - \sum \iota ( \underline{f} ( v^{i} ))  \nabla_{A, v_{i}}  \mathfrak{b}  \\
&\qquad \qquad  
 + \rho (f (a)) (\mathfrak{a} + \mathfrak{b} ) 
 + \rho (f (\underline{\theta})) ( \mathfrak{a} + \mathfrak{b} ), \\
\end{split}
\label{eq:diffpD}
\end{equation*} 
where $a \in \Omega^{1} (\mathfrak{g}_{P}) , \underline{f} \in C^{r} 
( \mathfrak{gl} (T^{*}X))$, 
$\underline{\theta} \in C^{r} (\Lambda^{1} \otimes \C) $ and $\mathfrak{a} \in \Omega^{0} (\mathfrak{g}_{P}) , 
\mathfrak{b} \in \Omega^{0,2} (\mathfrak{g}_{P})$.

In Section \ref{sec:rankletwo}, 
we prove that there are no rank one or two solutions $(A, (\alpha ,
\beta))$ 
to the equations
\eqref{eq:p2VW1} and \eqref{eq:p2VW2} with $A$ irreducible, 
$\alpha - \bar{\alpha} =0$ and $( \alpha , 
\beta) \neq 0$ for a generic choice of perturbation parameters.

\section{Generic smoothness for the rank three case}
\label{sec:rankthree}

In this section, 
we prove Proposition \ref{th:rank3}. 
In order to do that we consider  
the {\it parametrized moduli space}, and prove that it is a smooth
manifold (Proposition \ref{prop:rkth}). 
Then Proposition \ref{th:rank3} follows from Proposition
\ref{prop:rkth}.

\subsection{Parametrized moduli space}
\label{sec:univmod}

Let $X$ be a compact symplectic four-manifold with symplectic form
$\omega$, and let $P$ be a principal $G$-bundle over $X$. 
From now on, $G$ is either $SU(2)$ or $SO(3)$. 
We take an almost complex structure $J$ compatible with 
the symplectic form $\omega$.

We denote by $\mathcal{A}^{2}_{k} (P)$ 
the $L^{2}_{k}$-completion of the space of connections on $P$, 
and by $\mathcal{G} (P) = \mathcal{G}^{2}_{k+1} (P)$ 
the $L^{2}_{k+1}$-completion of the gauge group. 
We set 
$$ \mathcal{C}(P) 
:= \mathcal{A}^{2}_{k} (P) 
\times L^{2}_{k} (\mathfrak{g}_{P} \otimes 
( \Lambda^{0,0} \oplus \Lambda^{0,2} )) , $$
and 
$ \mathcal{P}_1 
:= 
C^{r} (GL ( \Lambda^{0,2})) 
\times C^{r} (GL (\Lambda^{0,0}))
\times C^{r} (GL ( \Lambda^{2,2})) \times 
 C^{r} (\Lambda^1 \otimes \C) $. 
This $\mathcal{P}_1$ is the parameter space for the perturbation described in Section \ref{sec:perteq}.  
We denote the quotient $\mathcal{C}(P) / \mathcal{G}(P)$ by $\mathcal{B}
(P)$.

We define 
$$ s : \mathcal{C} (P) \times \mathcal{P}_1 \to 
L^{2}_{k-1} \left( \mathfrak{g}_{P} \otimes \Lambda^{0,1}  \right) 
\times 
 L^{2}_{k-1} \left( \mathfrak{g}_{P} \otimes ( \Lambda^{0,2} \oplus \Lambda^{
1,1 } ) \right) 
 $$
by $
s \left(  A , (\alpha , \beta) ,  \tau , \theta  \right) 
:= ( s_{1} ( A , (\alpha , \beta) ,  \tau , \theta) , s_{2} ( A ,
(\alpha , \beta) ,  \tau , \theta) )$, where 
\begin{equation*}
\begin{split}
s_{1} ( A , (\alpha , \beta) ,  \tau , \theta) 
&:= \bar{\partial}_{A} \alpha  
 + \bar{\partial}_{A}^{*} \beta 
+ \rho (\theta) (\alpha + \beta )  , \\ 
s_{2} ( A , (\alpha , \beta) ,  \tau , \theta) &:= 
F_{A}^{0,2} + \frac{1}{2} \tau_{1} [ \alpha , \beta ]  
+ \Lambda F_{A}^{1,1} \wedge \omega \\ 
 &\qquad \qquad \qquad + \frac{1}{2} \tau_{2} [\alpha , \alpha^{*} ] \wedge
 \omega  + \Lambda \tau_{3} [ \beta , \beta^{*} ] .
\end{split}
\end{equation*}
This is a $\mathcal{G}(P)$-equivariant map, where the action of $\mathcal{G} (P)$ on $\mathcal{P}_1$ is taken to be trivial. 
Here $\rho : T^{*} X \otimes \C \to \text{Hom}_{\C} 
(\Lambda^{0,0} \otimes
\Lambda^{0,2} , \Lambda^{0,1})$ is the Clifford multiplication map, 
and $\tau := ( \tau_1 , \tau_2 , \tau_3 ) \in \mathcal{P}_1$. 
We say $M (P) := s^{-1} (0) / \mathcal{G} (P) \subset
\mathcal{B} (P) \times \mathcal{P}_1 $ 
the {\it parametrized moduli space}.

We denote by $\mathcal{B}^{*}_{\diamond} (P)$ gauge equivalence classes of
pairs $(A , (\alpha , \beta)) \in 
\mathcal{C} (P)$
with $A$ irreducible, $\alpha - \bar{\alpha} =0$ and $ \alpha + \beta
$ being of rank three. 
We set $M^{*}_{\diamond} (P) := M (P)\cap 
 ( \mathcal{B}^{*}_{\diamond} (P) \times \mathcal{P}_1)$. 
We then have the following:

\begin{proposition}
The zero set $s^{-1} (0)$ in $\mathcal{B}^{*}_{\diamond} (P) \times \mathcal{P}_1$ is
 regular, in particular,  the parametrized  
moduli space $M^{*}_{\diamond} (P)$ is a smooth Banach
 submanifold of $\mathcal{B}^{*}_{\diamond} (P) \times \mathcal{P}_1$. 
\label{prop:rkth}
\end{proposition}

We prove Proposition \ref{prop:rkth} in Section \ref{sec:rkth}. 
Proposition \ref{th:rank3} follows from Proposition \ref{prop:rkth} as
described below.

\vspace{0.2cm}

\noindent
{\it Proof of Proposition \ref{th:rank3}.}
Note that  $s$ is a Fredholm section if it is restricted to
$\mathcal{B} (P) \times \{ (\tau , \theta )\}$ for a perturbation
parameter $(\tau , \theta)$.  
Thus, by the Sard--Smale theorem (\cite[Prop~4.3.11]{DK}), 
there exists a first category subset
$\mathcal{P}_1'$ such that the zero set of $s$ in
$\mathcal{B}^{*}_{\diamond} (P)$ is 
regular for $(\tau , \theta ) \in \mathcal{P}_1 \setminus \mathcal{P}_1'$. 
Hence,  
$\mathcal{M}^{*}_{\diamond} 
(\tau, \theta) =  s^{-1} (0) \cap \mathcal{B}^{*}_{\diamond} (P)$ 
is a smooth
manifold for generic $C^{r}$-parameters $(\tau, \theta)$. 
\qed

\subsection{Proof of Proposition \ref{prop:rkth}}
\label{sec:rkth}

In this section, we prove Proposition \ref{prop:rkth}. 
We follow an argument by Feehan \cite[\S 2.2]{F} (see also \cite[\S 4.3.5]{DK}). 
First we consider the linearisation $D s= (D s_{1} , D s_{2}) 
:  
L^{2}_{k} (\mathfrak{g}_{P} \otimes \Lambda^{0,1}) \times 
L^{2}_{k} 
(\mathfrak{g}_{P} \otimes ( \Lambda^{0,0} \oplus \Lambda^{0,2})) 
\times  \mathcal{P}_1 
\to L^{2}_{k-1} (\mathfrak{g}_{P} \otimes \Lambda^{0,1}) \times L^{2}_{k-1} 
(\mathfrak{g}_{P} \otimes ( \Lambda^{1,1}  \oplus \Lambda^{0,2} ))$ of $s$ 
at $(A , (\alpha, \beta), \tau, \theta ) \in s^{-1} (0)$, where 
\begin{equation*}
\begin{split}
D s_1 
( (\underline{\tau} , \underline{\theta}), a , ( \mathfrak{a},
 \mathfrak{b})) 
&= \bar{\partial}_{A}  \mathfrak{a} + \bar{\partial}_{A}^{*}
 \mathfrak{b} 
+ \rho (\theta) (  \mathfrak{a} +  \mathfrak{b} )
 + \rho(\underline{\theta}) ( \alpha + \beta ) ,\\
D s_{2} ( (\underline{\tau} , \underline{\theta}), a , ( \mathfrak{a} ,
  \mathfrak{b} )) 
&= \bar{\partial}_A a + \partial_A a 
+ \frac{1}{2} \tau_{1} \left( [  \mathfrak{a} , \beta ]  
    + [ \alpha ,  \mathfrak{b} ] \right)  
  + \frac{1}{2} \underline{\tau_1} \tau_{1} [\alpha , \beta] \\ 
 & \qquad +  \frac{1}{2} \tau_2
 \left( [ \alpha ,  \mathfrak{a}^{*}] + [ \mathfrak{a} , \alpha^{*} ] \right)
 \wedge \omega   
 + \frac{1}{2} \underline{\tau_2} \tau_2 \left( [ \alpha , \alpha^{*} ] \right)
 \wedge  \omega \\   
 & \qquad \qquad 
   + \Lambda \tau_3 \left( [ \beta ,  \mathfrak{b}^{*} ] 
   +  [  \mathfrak{b} , \beta^{*} ] \right)  
   + \Lambda \underline{\tau_3} \tau_3 \left( [ \beta , \beta^{*} ] \right). \\
\end{split}
\end{equation*}

We then suppose for a contradiction that there exists $(\delta , v) 
\in C^{0} ( \mathfrak{g}_{P} \otimes \Lambda^{0,1}) \times 
 C^{0} ( \mathfrak{g}_{P} \otimes ( \Lambda^{1,1} \oplus \Lambda^{0,2}))
 $ with $(\delta , v) \neq 0$ 
such that 
\begin{equation}
\left\langle 
D s_{1} 
 (  
 a , (  \mathfrak{a} ,  \mathfrak{b}) ,\underline{\tau} , \underline{\theta} )  , 
\delta \right\rangle_{L^2} 
 = 0, 
\quad 
\left\langle 
D s_{2} 
 (  a , (  \mathfrak{a} ,  \mathfrak{b}), \underline{\tau} , \underline{\theta} )  , 
 v \right\rangle_{L^2} = 0 . 
\label{eq:ds12de}
\end{equation}
By setting $ (  \mathfrak{a} ,  \mathfrak{b}
)= 0$ in the first equation of \eqref{eq:ds12de}, 
we get 
\begin{equation}
 \left\langle \rho ( \underline{\theta} ) (\alpha + \beta ) , 
\delta \right\rangle_{L^2} = 0 
\label{eq:1} 
\end{equation}
for $\underline{\theta} \in C^{r} ( \Lambda^{1} \otimes \C)$.

\begin{lemma}
Assume that $ \alpha + \beta \in C^{0} ( 
\mathfrak{g}_{P} \otimes ( \Lambda^{0,0} \oplus 
\Lambda^{0,2}) )$ and 
$\delta \in C^{0} ( \mathfrak{g}_{P} \otimes \Lambda^{0,1} )$ satisfy \eqref{eq:1}. 
Then $\alpha + \beta $ and $\delta $ have orthogonal images in
 $\mathfrak{g}_{P}$ at each point of $X$, in particular,  
$$ \text{\rm rank}_{\R}\, (\alpha + \beta ) (x) + \text{\rm rank}_{\R}\, 
\delta (x) \leq 3 $$
at each point $x \in X$. 
\label{lem:r3}
\end{lemma}

\begin{proof}
In \eqref{eq:1}, $\underline{\theta} \in C^{r} (\Lambda^1 \otimes \C)$ 
is arbitrary, thus, we get the point-wise identity 
$$  \left\langle \rho ( \underline{\theta}_{x} ) 
(\alpha + \beta ) (x),  \delta  (x) \right\rangle_{x} = 0 $$
for all $\underline{\theta}_{x} \in (T^{*} X)_{x} \otimes \C$.

We then recall the following.

\begin{lemma}[\cite{F}, Lem.~2.3]
Let $U$ and $V$ be complex vector spaces with 
$\dim U \leq \dim V$, and let $W$ be a real vector
 space.  
We take $M \in  U^{*} \otimes_{\R} W$ 
and $N \in V^{*} \otimes_{\R} W$. 
Then, if $ \langle M P , N \rangle_{V^{*} \otimes_{\R} W} =0$ for all $P \in
 \text{\rm Hom}_{\C}\, (V ,U)$, 
we get $\text{\rm Ran} \, M \perp \text{\rm Ran} \, N$ in $W$, in
 particular, 
$ \text{\rm rank}_{\R}\, M + \text{\rm rank}_{\R}\, N \leq \dim_{\R} W $.
\label{prop:F1}
\end{lemma}

Since $\rho$ gives a complex linear isomorphism
$$ (T^{*} X)_{x} \otimes_{\R} \C \to 
\text{Hom}_{\C} (\Lambda^{0,1} \oplus \Lambda^{0,2} , \Lambda^{0,1})_{x} ,$$
we can invoke Lemma \ref{prop:F1} to obtain the assertion. 
\end{proof}

As $(A , (\alpha , \beta)) \in \mathcal{C}^{*}_{\diamond} (P)$ 
and $\alpha + \beta$ is $C^r$ for some $r$, 
there is a non-empty open subset $U \subset X$ on which 
$\text{rank}_{\R} \, (\alpha + \beta ) (x)  = 3$ for all $x \in U$. 
Then Lemma \ref{lem:r3} implies that $\text{rank}\, \delta (x) = 0 $ for
all $x \in U$, namely, $\delta \equiv 0$ on $U$.

In a similar way, by setting $(  a , (  \mathfrak{a} ,
 \mathfrak{b} )) = 0$ in the second equation of \eqref{eq:ds12de}, we get  
\begin{equation}
  \left\langle 
\frac{1}{2}  \underline{\tau_1} \tau_1 [ \alpha , \beta ]  
  + \frac{1}{2} \underline{\tau_2} \tau_2 [ \alpha , \alpha^{*} ] \wedge \omega 
 + \Lambda \underline{\tau_3} \tau_3 [\beta , \beta^{*}]  ) , v
 \right\rangle_{L^2 (X)} = 0  
\label{eq:2}
\end{equation}
for all $\underline{\tau_{1}} \in C^{r} (\mathfrak{gl} (\Lambda^{0,2})), 
\underline{\tau_{2}} \in C^{r} (\mathfrak{gl} (\Lambda^{0,0}))$ 
and $\underline{\tau_{3}} \in C^{r} (\mathfrak{gl} (\Lambda^{2,2}))$.

\begin{lemma}
If $v \in C^{0} ( \mathfrak{g}_{P} \otimes ( \Lambda^{1,1} \oplus \Lambda^{0,2}))$ and 
$ \alpha + \beta 
\in C^{0} (\mathfrak{g}_{P} \otimes ( \Lambda^{0,0} \oplus \Lambda^{0,2}))$ satisfy 
\eqref{eq:2}, then 
$v$ and $\frac{1}{2} \tau_1 [ \alpha , \beta ] + \frac{1}{2} \tau_2 
[ \alpha , \alpha^{*} ] \wedge
 \omega 
 + \Lambda \tau_3 [ \beta , \beta^{*}]  
\in \text{\rm Hom} \, ( (\Lambda^{1,1} \oplus \Lambda^{0,2} )^{*} , \mathfrak{g}_{P})$ have
 orthogonal images in $\mathfrak{g}_{P}$ at each point in X, in particular, 
$$ \text{\rm rank}_{\R}\, v (x) 
  + \text{\rm rank}_{\R} \left( \frac{1}{2} [ \alpha , \beta ] + \frac{1}{2}[ \alpha ,
 \alpha^{*} ] \wedge \omega  + \Lambda [
  \beta , \beta^{*}]  \right) (x) \leq 3$$
at each $x \in X$.
\label{lem:rankb3} 
\end{lemma}

\begin{proof}
As $
\underline{\tau_{1}} \in C^{r} (\mathfrak{gl} (\Lambda^{0,2})), 
\underline{\tau_{2}} \in C^{r} (\mathfrak{gl} (\Lambda^{0,0}))$ and 
$\underline{\tau_{3}} \in C^{r} (\mathfrak{gl} (\Lambda^{2,2})) $ are 
 arbitrary, we get the following point-wise identity:  
\begin{equation*}
\left\langle \left(
\frac{1}{2}  \underline{\tau_1} \tau_1 [ \alpha , \beta ]  
  + \frac{1}{2} \underline{\tau_2} \tau_2 [ \alpha , \alpha^{*} ] \wedge \omega 
 + \Lambda \underline{\tau_3} \tau_3 [\beta , \beta^{*}]  \right)  (x) , v (x)
 \right\rangle_{x} = 0  
\end{equation*}
for all $
\underline{\tau_{1}} (x) \in C^{r} (\mathfrak{gl} (\Lambda^{0,2} |_{x})), 
\underline{\tau_{2}} (x) \in C^{r} (\mathfrak{gl} (\Lambda^{0,0} |_{x})) , 
\underline{\tau_{3}} (x) \in C^{r} (\mathfrak{gl} (\Lambda^{2,2} |_{x}))$ 
and for all $x \in X$. 
Then we again invoke Lemma \ref{prop:F1} to obtain the assertion. 
\end{proof}

The following is due to Mares \cite[\S 4.1.1]{BM}.

\begin{lemma}[\cite{BM}]
Let $ (A, \alpha + \beta)$ be an irreducible solution to the equation, 
and let $x \in X$. Then 
$\text{\rm rank}_{\R}\, \left(
\frac{1}{2} [ \alpha , \beta ] + \frac{1}{2} [ \alpha , \alpha^{*} ] \wedge \omega  +
 \Lambda [
  \beta , \beta^{*}]  \right) (x) =
	3$ 
if and only if $\text{\rm rank}_{\R}\, ( \alpha +  \beta) (x) = 3$.
\label{lem:BM}
\end{lemma}

From Lemma \ref{lem:BM}, if $\text{rank}\, ( \alpha +  \beta) (x) = 3$
for all $x \in U$, then 
$\text{rank}_{\R} \,  \frac{1}{2} [ \alpha , \beta ] + \frac{1}{2} [ \alpha ,
\alpha^{*} ] \wedge \omega 
+ \Lambda [ \beta , \beta^{*}] ) (x) = 3$ 
for all $x \in U$. 
Thus Lemma \ref{lem:rankb3} implies $\text{rank}\, v (x) =0$ for all $x \in
U$. 
Therefore, $v \equiv 0$ on $U$. 
Hence $( \delta , v) \equiv 0$ on $U$.  
Thus by unique continuation for the Laplacian $(Ds) (Ds)^{*}$ implies that 
$( \delta , v ) \equiv 0 $ on the whole of $X$. This is a
contradiction. 
\qed

\section{Non-existence of rank one and two cases}
\label{sec:rankletwo}

In this section, we prove Proposition \ref{th:norot}. 
Except modifications stated as Proposition \ref{prop:rankot} 
in Section \ref{sec:rankot} and Proposition \ref{prop:psurj} in 
Section \ref{sec:psurj}, the proof goes in a similar way to the case 
for the $PU(2)$-monopole equations by Feehan \cite{F}. 
In Section \ref{sec:term}, we introduce some terminology  
and a version of the Sard--Smale theorem 
from \cite{F}, which we use in the later sections.  
We give a characterization of the rank one and two sections in Section \ref{sec:rankot}. 
In Section \ref{sec:psurj}, we prove a surjectivity of some linear operator.  
We then prove Proposition \ref{th:norot} in Section \ref{sec:norankot} by using the Sard--Smale
theorem.

\subsection{Banach spaces, Fredholm operators and the Sard--Smale theorem}
\label{sec:term}

Let $V$ be a Banach space. 
For each $k \geq 1$, we define the infinite dimensional 
Grassmannian 
by 
$$ 
\mathbb{G}_{k} (V) 
:= 
\{ K \subset V \, : \, 
K \text{ is a $k$-dimensional subspace of $V$}  \}.$$
We write $\mathbb{P} (V) = \mathbb{G}_{1} (V)$. 
We also define the infinite dimensional flag manifold by 
$$ 
\mathbb{F}_{k} (V) 
:= \{ ( \ell ,  K ) \in 
 \mathbb{P} (V) \times \mathbb{G}_{k} (V) \, : \, \ell \subset K \} . $$
We denote the projections by 
$\pi_{1} : \mathbb{F}_{k} (V) \to \mathbb{P} (V) $ and 
$\pi_{2} : \mathbb{F}_{k} (V) \to \mathbb{G}_{k} (V)$. 
Note that both $\pi_1$ and $\pi_2$ are submersions (see Claims 4.2 and 4.3 in
\cite{F}).

We next consider a smooth submanifold 
$Z \in \mathbb{P} (V)$. 
We set  $I_{k} (Z) := \pi_{2} (\pi_{1}^{-1} (Z)) \subset
\mathbb{G}_{k} (V)$. 
As $\pi_1$ is a submersion, $\tilde{I}_{k} (Z) := \pi_{1}^{-1} (Z)
\subset \mathbb{F}_{k} (V)$ is a smooth submanifold. 
Note that, however, $I_{k} (Z)$ is not necessarily a submanifold.

\paragraph{Space of Fredholm operators.}

Let $V_1, V_2$ be Banach spaces. 
We denote by $\text{Fred}_{n} (V_1 , V_2)$ the space of bounded Fredholm
operators of index $n$ in the Banach space of the bounded operators. 
In our case, we take $V_1 := 
L^{2}_{k} ( \mathfrak{g}_{P} \otimes ( \Lambda^{0,0}_{I} \oplus
\Lambda^{0,2})  )$, 
where $\Lambda^{0,0}_{I} := 
\{ \alpha \in \Lambda^{0,0} \, : \, \alpha - \bar{\alpha} =0 \}$ and 
$V_2 := L^{2}_{k-1} ( 
\mathfrak{g}_{P}\ \otimes \Lambda^{0,1} )$ in the subsequent sections. 
We define 
$$\text{Fred}_{k,n} 
:= 
\{ A \in \text{Fred}_{n} (V_1 , V_2 ) \, : \, 
\dim_{\R} \ker A = k \}. $$
We also define a map 
$$ \pi : \text{Fred}_{k,n} (V_1 , V_2) \to \mathbb{G}_{k} 
(V_1)  $$ 
by $A \mapsto \ker A$. 
This is smooth, and a submersion (\cite[Lem.~4.5]{F}). 
We then define 
the following flag manifold for each $\text{Fred}_{k,n} (V_1, V_2)$:  
\begin{equation*}
\begin{split}
\text{Flag}_{k,n} 
 ( V_1, V_2) := \{ 
( \ell , A ) \in \mathbb{P} (V_1) 
\times \text{Fred}_{k,n}  (V_1, V_2 ) \, 
: \, \ell \in \ker A \}.
\end{split} 
\end{equation*}
This $\text{Flag}_{k,n} (V_1 , V_2)$ is a smooth submanifold of $\mathbb{P} (V_1) \times 
\text{Fred}_{k,n} 
(V_1 ,V_2)$ and the canonical map 
$\varpi : \text{Flag}_{k,n} (V_1 , V_2) \to \mathbb{F}_{k} (V_1)$ is a
 submersion (see \cite[Lem.~4.6]{F}).

\paragraph{The Sard--Smale theorem.}

We state a version of the Sard--Smale theorem from \cite{F}.

\begin{proposition}[\cite{F}, Prop.4.12]
Let $\mathcal{C}, \mathcal{P}, \mathcal{F}$ be $C^{\infty}$-Banach
manifolds. 
Suppose that ${M} \subset \mathcal{C} \times \mathcal{P}$ is a
$C^{\infty}$-Banach submanifold, and the restriction $\pi_{{M},
\mathcal{P}} : {M} \to \mathcal{P}$ of the projection map
$\pi_{\mathcal{P}} : 
 \mathcal{C} \times \mathcal{P} \to \mathcal{P}$ is Fredholm. 
Let $\underline{v} : {M} \subset 
\mathcal{C} \times \mathcal{P} \to \mathcal{F}$ be a $C^{\infty}$-map
which is transverse to a $C^{\infty}$-Banach submanifold $\mathcal{J}
\subset \mathcal{F}$. 
Then there exists a first category subset $\mathcal{P}' \subset \mathcal{P}$
 such that the following holds. 
For all $p \in \mathcal{P} \setminus \mathcal{P}'$,  
\begin{itemize}
\item $M : =\pi_{{M} , \mathcal{P}}^{-1} (p)$ is a
      $C^{\infty}$-manifold of dimension $\text{\rm ind}\, (\pi_{{M}
      , \mathcal{P}})_{p} < \infty$ \!\!$;$ 
\item $v := \underline{v} (\cdot , p) : 
{M} \to \mathcal{F}$ is transverse to the submanifold $\mathcal{J}
      \subset \mathcal{F}$\!\! $;$ and 
\item 
$Z := v^{-1} (\mathcal{J}) \subset M$ is a $C^{\infty}$-submanifold of
      codimension 
$\text{\rm codim}\, (Z , M) = \text{\rm codim}\, (\mathcal{J}, \mathcal{F})$. 
\end{itemize}
\label{prop:SS}
\end{proposition}

We use this to prove Proposition \ref{th:norot} in Section
\ref{sec:norankot}.

\subsection{Rank one and two loci}
\label{sec:rankot}

We take $k \geq 4$ so that 
$V_1 = L^{2}_{k-1} (\mathfrak{g}_{P} \otimes (\Lambda_{I}^{0,0} \oplus
\Lambda^{0,2})) 
\subset C^{0} (\mathfrak{g}_{P} \otimes (\Lambda_{I}^{0,0} \oplus \Lambda^{0,2}))$. 
We think of $C^{0} (\mathfrak{g}_{P} \otimes (\Lambda_{I}^{0,0} \oplus
\Lambda^{0,2})) $
as
$C^{0} \left( \text{Hom}_{\R} \left( 
(\Lambda^{0,0}_{I} \oplus \Lambda^{0,2} 
 )^{*}
, \mathfrak{g}_{P} \right) \right)$, and define a determinant map 
$$ h: 
C^{0} (\mathfrak{g}_{P} \otimes (\Lambda_{I}^{0,0} \oplus
\Lambda^{0,2})) 
\to C^{0} \left( \det (\Lambda^{0,0}_{I} \oplus \Lambda^{0,2})
\otimes \det (\mathfrak{g}_{P}) \right)$$ 
by $\varphi \in C^{0} (\mathfrak{g}_{P} \otimes (\Lambda_{I}^{0,0}
\oplus \Lambda^{0,2})) 
\mapsto \det \varphi$, 
where $\det ( \Lambda^{0,0}_{I} \oplus \Lambda^{0,2}) 
= \Lambda^3 ( \Lambda^{0,0}_{I} \oplus \Lambda^{0,2}) 
$ and $\det (\mathfrak{g}_{P}) = \Lambda^3 \mathfrak{g}_{P}$.  
Then $\varphi \in V_1$ with 
$\varphi \neq 0$ is of rank one or two if and only if $h (\varphi) =0$. 
We define 
$$ \mathcal{Z} :=  
\{ [\varphi] \in \mathbb{P} ( V_1  ) \, : \, 
h ( \varphi ) = 0 \},$$
where $[ \varphi ]$ is the line $\R \cdot \varphi \subset 
V_1 $. 
We denote by $\mathcal{Z}'$ the smooth part of $\mathcal{Z}$.

As in the case of the $PU(2)$-monopole equations \cite[Lem~4.7]{F}, 
one obtains the following:

\begin{proposition}
Let $[\varphi] \in \mathcal{Z}$. 
We assume that $\{ \varphi \neq 0 \}$ is a dense open subset of $X$. 
Then the determinant map $h : C^{0} ( \mathfrak{g}_{P} \otimes (\Lambda_{I}^{0,0}
 \oplus \Lambda^{0,2}) ) 
\to C^{0} \left( \det (\Lambda^{0,0}_{I} \oplus \Lambda^{0,2})
\otimes \det (\mathfrak{g}_{P}) \right)$ vanishes transversely at $\varphi$, 
and $[ \varphi ]$ is a smooth point of $\mathcal{Z}$. 
In addition, 
the tangent space $T_{[\varphi]} \mathcal{Z}$ has both infinite dimension and
infinite codimension in $T_{[\varphi]} \mathbb{P} ( V_1 )$, in
 particular, we have  
$ 
\text{\rm codim}\, (\mathcal{Z}' , \mathbb{P} ( V_1 ) ) = \infty $.  
\label{prop:rankot}
\end{proposition}

\begin{proof}
We take a local orthonormal frame 
$\{ \phi_{1} , \phi_{2} , \phi_{3} \}$ for
$\mathfrak{g}_{P}$,  and local orthonormal frame 
$\{ e_{1} , e_{2} , e_{3} \}$ for
$\Lambda^{0,0}_{I} \oplus \Lambda^{0,2}$ on an open subset $U \subset X$
so that  
$\varphi = 
\left( 
\begin{matrix}
\varphi_{11} & \varphi_{12} & \varphi_{13} \\
\varphi_{21} & \varphi_{22} & \varphi_{23} \\
\varphi_{31} & \varphi_{32} & \varphi_{33} \\
\end{matrix}
\right) $. 
Then the differential of $h$ at $\varphi$ with respect to these frame 
is given by 
\begin{equation*}
\begin{split}
&(D h)_{\varphi} ( \underline{\varphi}) 
= \\
&\sum_{\sigma \in \mathfrak{S}_{3}}
\left\{ \text{sgn} (\sigma) \left( 
\underline{\varphi_{1 \sigma(1)}} \varphi_{2 \sigma (2)} \varphi_{3
 \sigma (3)} 
+ \varphi_{1 \sigma(1)} \underline{\varphi_{2 \sigma (2)}} \varphi_{3
 \sigma (3)} 
+ \varphi_{1 \sigma(1)} \varphi_{2 \sigma (2)} 
 \underline{\varphi_{3 \sigma (3)}} 
\right)\right\} , 
\end{split}
\end{equation*}
where 
$
\underline{\varphi}
= 
\left( 
\begin{matrix}
\underline{\varphi_{11}} & \underline{\varphi_{12}} &
 \underline{\varphi_{13}} \\
\underline{\varphi_{21}} & \underline{\varphi_{22}} &
 \underline{\varphi_{23}} \\
\underline{\varphi_{31}} & \underline{\varphi_{32}} &
 \underline{\varphi_{33}} \\
\end{matrix}
\right) 
\in 
C^{\infty} (U , \mathfrak{gl} (3, \R))$.

We now suppose for a contradiction that there exists $\psi \in 
\text{coker} (D h)_{\varphi}$ so that  
$\langle 
(Dh)_{\varphi} ( \underline{\varphi} ) , \psi \rangle_{L^2} =0$ 
for all $ \underline{\varphi} \in C^{0} (V_1)$. 
From the assumption, $\{ \varphi \neq 0\}$ is dense in $U$, 
so the union of the complements of each zero set of $\varphi_{ij}'s$ is
 a dense open subset of $U$, hence we get $\psi \equiv 0$ on $U$. 
Since $U$ was arbitrary, $\psi \equiv 0$ on $X$. This is a
 contradiction. 
\end{proof}

We denote by $M^{*,0} (P)$ the parametrized moduli space for the perturbed
 Vafa--Witten equations  
\eqref{eq:p2VW1} and \eqref{eq:p2VW2}
with
 $A$ irreducible, $\alpha - \bar{\alpha} =0$ and $(\alpha ,\beta ) \neq
 0$. 
From Proposition \ref{prop:rankot}, we get the following:

\begin{Corollary}
If $(A, \varphi = (\alpha , \beta) , \tau , \theta)$ is in 
$M^{*,0} (P)$  
so that $h (\varphi) =0$, then 
$[ \varphi ]$ is a smooth point of $\mathcal{Z} \subset \mathbb{P} 
( V_1 )$, that is,  
$\pi ( M^{*,0} (P)) \subset \mathcal{Z}'$, 
where $\pi : M^{*,0} (P) \to 
\mathbb{P} ( V_1 )$ is the
 projection. 
\label{cor:zsm} 
\end{Corollary}

For each $k \geq n$, we now define 
$$ \tilde{I}_{k} 
(\mathcal{Z}) 
:= 
\pi_{1}^{-1} (\mathcal{Z}) \subset \mathbb{F}_{k} 
(V_1) , 
$$
and $I_{k} (\mathcal{Z}) 
:= \pi_{2} (\tilde{I}_{k} (\mathcal{Z})) \subset \mathbb{G}_{k} 
(V_1 )$. 
By Corollary \ref{cor:zsm}, we only consider $I_{k} (\mathcal{Z}')$ and
$\tilde{I}_{k} (\mathcal{Z}')$ for our purpose. 
As $\pi_1 : \mathbb{F}_{k} (V) \to \mathbb{P} (V)$ is a submersion
(\cite[Claim 4.2]{F}), 
$\tilde{I}_{k} (\mathcal{Z}')$ is a smooth submanifold of
$\mathbb{F}_{k} 
( V_1 )$ with codimension 
\begin{equation*}
\text{codim} 
( \tilde{I}_{k} (\mathcal{Z}'), \mathbb{F}_{k} 
( V_1 ) )  
= \text{codim} (\mathcal{Z}' ,
 \mathbb{P} 
( V_1  )) = \infty .   
\end{equation*}
We put $J_{k} (\mathcal{Z}') := 
\pi^{-1} (I_{k} (\mathcal{Z}')) \subset 
\text{Fred}_{k,n} ( V_1 , V_2 ) 
$, 
where $\pi : 
\text{Fred}_{k,n} ( V_1 ,V_2 )  \to 
\mathbb{G}_{k} ( V_1  )$. 
We now define the {\it rank one and two loci}
$\tilde{J}_{k} (\mathcal{Z'}) := 
\varpi^{-1} ( \tilde{I}_{k} (\mathcal{Z'})) 
\subset \text{Flag}_{k,n} ( 
V_1 , V_2) $, where $\varpi : \text{Flag}_{k,n} (V_1 , V_2) \to
\mathbb{F}_{k} (V_1)$ is the canonical map.  
As $\varpi : \text{Flag}_{k,n} ( 
V_1 , V_2 ) \to 
\mathbb{F}_{k} (V_1 )$ 
is a submersion (\cite[Lem~4.6]{F}),  
the rank one and two loci $\tilde{J}_{k} (\mathcal{Z}')$ is a smooth submanifold, and  we get 
\begin{equation*}
 \text{codim} 
( \tilde{J}_{k} (\mathcal{Z}') , 
\text{Flag}_{k,n} ( 
V_1 , V_2  ) )
= 
\text{codim} 
( \tilde{I}_{k} (\mathcal{Z}') , \mathbb{F}_{k} 
(V_1) )= \infty .
\end{equation*}

\subsection{A surjectivity}
\label{sec:psurj}

In this section and the upcoming one, 
we take 
$\mathcal{P}_{2} 
:= 
C^r (GL (T^{*} X )) \times 
C^r ( \Lambda^{1} \otimes \C )
$ as the perturbation parameter space, 
since the perturbation parameter $\tau = (\tau_1 , \tau_2 , \tau_3)$ is
not needed in the proof of Proposition \ref{th:norot}.

We denote by $\mathcal{C}^{*} (P)$ the set of pairs $(A, (\alpha ,
\beta)) \in \mathcal{C} (P)$ with $A$ irreducible and $\alpha -
\bar{\alpha} =0$. 
As in \cite[\S 4.4]{F} (see also \cite[\S 4.3.3]{DK}), we consider the {\it period map}  
$$v : \mathcal{C}^{*} (P) 
\times \mathcal{P}_2 \to \text{Fred}_{n} 
(V_1 , V_2) ,$$
defined by 
$(A, (\alpha , \beta) , f , \theta) \mapsto 
\left( \bar{\partial}_{A, (f, \theta)} 
+ \bar{\partial}^{*}_{A, (f, \theta)} \right)$.  
The differential of $v$ at $(A, (\alpha ,\beta), f ,\theta)$ 
$$ (D v)_{(A, (\alpha ,\beta), f ,\theta)} 
: T_{(A, (\alpha ,\beta))} \mathcal{C}^{*} (P) 
\oplus T_{(f,\theta)} \mathcal{P}_2 \to 
\text{Hom}_{\R} (V_1 ,V_2) $$ 
is given by 
$(a , (\mathfrak{a} , \mathfrak{b}) , \underline{f} , \underline{\theta})
 \mapsto 
D \left( \bar{\partial}_{A,(f,\theta)}  + \bar{\partial}^{*}_{A,
(f,\theta)} \right)_{(A, (f , \theta))} (a , \underline{f} , \underline{\theta})$.

We denote by $\mathcal{C}^{*,0} (P)$ the set of pairs $(A , (\alpha ,
\beta) ) \in \mathcal{C} (P)$ with $A$ irreducible, $\alpha -
\bar{\alpha} =0$ and $(\alpha , \beta) \neq 0$, 
and by $\mathcal{B}^{*,0} (P)$ the quotient $\mathcal{C}^{*, 0} (P) /
\mathcal{G} (P) $. 
We set $M^{*,0} (P) = M (P)\cap (\mathcal{B}^{*,0} (P) \times \mathcal{P}_{2})$, 
where $M (P)$ is the parametrized moduli space for the equations \eqref{eq:p2VW1} and
\eqref{eq:p2VW2}. 
In this section, we prove the following:

\begin{proposition}
Let $(A , (\alpha , \beta) , f , \theta) \in M^{* , 0} (P)$. 
Then, the following is surjective. 
\begin{equation*}
(Dv)_{(A , (\alpha ,\beta) , f , \theta)} 
(0, \cdot) : 
  \{ 0 \} \oplus T_{(f, \theta )} \mathcal{P}_2  \to 
 T_{ v (A , (\alpha , \beta) , f , \theta) } \text{\rm Fred}_{n} \, 
 ( V_1 , V_2  ) .
\end{equation*}
\label{prop:psurj}
\end{proposition}

\begin{proof}

A proof here is a modification of that of \cite[Prop.~4.9]{F}.  
First we prove the following lemma:

\begin{lemma}
Assume that $(A , ( \alpha ,\beta)))$ is 
a solution to the Vafa--Witten equations \eqref{eq:p2VW1} and
\eqref{eq:p2VW2} with $A$ irreducible and  $(\alpha , \beta) \neq 0$ for 
some perturbation parameter $(f, \theta ) \in \mathcal{P}_2$. 
If $b \in \Omega^{0,0} 
 (X , \mathfrak{g}_{P}) \oplus \Omega^{0,2} (X , \mathfrak{g}_{P})$ 
and $d \in \Omega^{0,1} (X ,
 \mathfrak{g}_{P})$ satisfy 
$$  \left\langle D \left( \bar{\partial}_{A , (f , \theta)} 
+ \bar{\partial}_{A , (f , \theta) }^{*} \right)_{(A , ( f , \theta))} 
 ( \underline{f}, \underline{\theta} ) , d \otimes b^{*} \right\rangle_{L^2 (X)}
 = 0 $$
for all $(\underline{f} , \underline{\theta})$, then $d \otimes b^{*} \equiv
 0$ on $X$. 
\label{lem:k}
\end{lemma}

\begin{proof}
Suppose for a contradiction that $d \otimes b^{*} \neq 0$ on $X$. 
By varying $\underline{\theta}$, we see that $b$ and $d$ have orthogonal images in
 $\mathfrak{g}_{P}$ at each point $x \in X$ from Lemma \ref{lem:r3}. 
We then set $U := \{ b \neq 0 \} \cap \{ d \neq 0 \} \subset X$. 
Then either $b$ or $d$ defines a subbundle $\xi_1 \subset
 \mathfrak{g}_{P}$ 
on $U$ of $\text{rank}_{\R} =2$. 
We define $\xi_2 := \xi_{1}^{\perp} \subset \mathfrak{g}_{P}|_{U}$ so that 
$\mathfrak{g}_{P} |_{U} = \xi_{1} \oplus \xi_{2}$. 
The connection $A|_{U}$ on $\mathfrak{g}_{P} |_{U}$ also splits into the
 following form: 
\begin{equation*}
A =
\left( 
\begin{matrix}
A_1 & - \chi^{*} \\
\chi & A_2 \\
\end{matrix} 
\right) ,
\end{equation*}
where $A_{i}$ is a connection on $\xi_{i}$ for $i =1,2$, 
and $\chi \in \Omega^{1} (U , \xi_2 \otimes \xi_{1}^{*})$ is the second
 fundamental form. 
As $(A, (\alpha , \beta))$ is irreducible and non-zero section from the
 assumption, 
$\chi \neq 0$ on $U \subset X$. 
We suppose that $b \in  
\Omega^{0,0} (U , \xi_1) \oplus 
\Omega^{0,2} (U , \xi_1) $.  
We then get 
\begin{equation*}
\begin{split}
&D \left( \bar{\partial}_{A, (f , \theta)} 
+ \bar{\partial}_{A, (f , \theta)}^{*} \right)_{(A , (f, \theta))}  
( \underline{f} , \underline{\theta} ) b  \\
&
\quad \qquad 
= \sum_{i=1}^{4} ( \underline{f} ( v^{i} ) ) \wedge \nabla_{A_1, v_{i}} b  
- \sum_{i=1}^{4} \iota ( \underline{f} ( v^{i} ) )  \nabla_{A_1, v_{i}} b 
 + \rho ( f ( \underline{\theta} )) b + \rho (\underline{f} (\chi) ) b . 
\end{split}
\end{equation*}
This turns out to be 
$$ \left\langle 
 \rho \left( \underline{f}_{x} ( \chi_{x}) \right) b_{x} , d_{x} \right\rangle_{x} = 0 $$
at each $x \in U$ and for all $\underline{f}_{x} \in \mathfrak{gl} (T^{*}
 X)_{x}$. 
Hence we get $d_x \otimes b_{x}^{*} =0$ at each $x \in U$ with 
$\chi_{x} \neq 0$. 
As $d_{x} \otimes b_{x} \neq 0$ for all $x \in U$ from the assumption, 
we get $\chi =0$, thus, $A|_{U}$ is reducible.

On the other hand, by a similar argument by Feehan--Lenes \cite[\S 5.3]{FL}, 
one can obtain that,  
if $A$ is reducible on a non-empty open subset $U \subset X$ 
and $(\alpha , \beta) \neq 0$, 
$A$ is reducible on $X$. 
This is a contradiction. 
Therefore, $U \subset X$ is empty and $d \otimes b^{*} \equiv 0$ on
 $X$. 
\end{proof}

We now suppose that $(D v)_{(A, (\alpha ,\beta) , f , \theta )} (0, \cdot)$ 
is not surjective. 
Then, there exist sections $ b \in L^{2}_{k} (V_1)$ 
and $d \in L^{2}_{k-1} (V_2 )$
 with $d \otimes b^{*} \not \equiv 0$ on $X$ such that 
$$ \left\langle 
 D \left( \bar{\partial}_{A, (f, \theta)} 
+ \bar{\partial}^{*}_{A , (f , \theta)} \right)_{(A , (f,\theta))} ( \underline{f} ,
 \underline{\theta}) b , d 
\right\rangle =0 . $$
Then, from Lemma \ref{lem:k}, we get $d \otimes b^{*} \equiv 0$. 
This is a contradiction. 
Therefore, 
$(D v)_{(A, (\alpha ,\beta) , f , \theta )} (0, \cdot)$ is surjective. 
\end{proof}

\subsection{No rank one and two loci}
\label{sec:norankot}

In this section, we prove Proposition \ref{th:norot}. 
As mentioned in the beginning of Section \ref{sec:rankletwo}, 
once Propositions \ref{prop:rankot} and \ref{prop:psurj} are obtained,
the proof of Proposition \ref{th:norot} goes along the same line with the case  for the
$PU(2)$-monopole equations \cite[\S 4.6]{F}. 
Hence we give it sketchily.

First note that the map $s  
: \mathcal{C}^{*,0} (P) \times \mathcal{P}_2 
\to L^{2}_{k-1}( \mathfrak{g}_{P} \otimes 
\Lambda^{0,1}  ) \times L^{2}_{k-1} 
( \mathfrak{g}_{P} \otimes ( \Lambda^{0,2} \oplus \Lambda^{1,1}) ) $ 
is right semi-Fredholm, 
namely, the differential has closed range and finite dimensional
cokernel. 
In particular,  
$$\mathbb{H}^{2}_{(A, (\alpha , \beta ) ,p)} 
:= \left( \text{Im} \, (D s ( \cdot , p) )
\right)^{\perp}_{(A, (\alpha ,\beta))}$$ 
is a  finite dimensional subspace
of $L^{2}_{k-1}( \mathfrak{g}_{P} \otimes \Lambda^{0,1}) \times L^{2}_{k-1} 
( \mathfrak{g}_{P} \otimes ( \Lambda^{0,2} \oplus \Lambda^{1,1})  ) $. 
We denote  by $\Pi_{(A, (\alpha ,\beta))}$ 
the $L^2$-orthogonal projection from $  L^{2}_{k-1}( \mathfrak{g}_{P} \otimes 
\Lambda^{0,1}  ) \times L^{2}_{k-1} 
( \mathfrak{g}_{P} \otimes ( \Lambda^{0,2} \oplus \Lambda^{1,1}) ) $ 
to the $\text{Im} \, (D s ( \cdot ,
p))_{(A, (\alpha ,\beta))}$.

Let $(c_0 , p_0) \in M^{*,0} (P) $. 
We consider the following composition:  
\begin{equation*}
\Pi_{(c_0 , p_0)} 
\circ s  
: \mathcal{B}^{*,0} (P) \times 
\mathcal{P}_2 
  \to 
\left( \mathbb{H}^{2}_{(c_0 , p_0)} \right)^{\perp}
\end{equation*} 
Then the differential at $(c_0 ,p_{0} )$
of $\Pi_{(c_0, p_0)} 
\circ s$ is surjective, 
in particular, 
it is surjective on some open neighbourhood 
$\mathcal{U}_{( c_0, p_0)}$ of 
$(c_0 , p_0)$ in $\mathcal{C}^{*,0} (P) \times
\mathcal{P}_2$. 
We set 
$$ \mathcal{T}_{(c_0 ,p_{0})} 
: = 
\mathcal{U}_{(c_0 ,p_{0})} 
 \cap \left( \Pi_{( c_0 , p_{0})}
 \circ s \right)^{-1} (0) \subset \mathcal{B}^{* ,0} (P) \times \mathcal{P}_2 .  
 $$
We denote by  $\pi_{\mathcal{T}, \mathcal{P}_2} : 
\mathcal{T}_{(c_0 , p_{0})} \to \mathcal{P}_2$ the projection,  
and define 
$\mathcal{T}_{( c_0 , p_{0})} |_{p}
:= \pi_{\mathcal{T}, \mathcal{P}_2}^{-1}(p) \cap \mathcal{T}_{(c_0 , p_0)}$. 
We then prove the following:

\begin{proposition}
There is a first-category subset $\mathcal{P}_2' \subset
 \mathcal{P}_2$, depending on $( c_0 ,
 p_{0})$ such that for any $p \in \mathcal{P}_2 \setminus \mathcal{P}_2'$, 
$\mathcal{T}_{(c_0 , p_{0})} |_{p}$ 
contains no $(A , (\alpha ,\beta ) ,p)$ with $\alpha + \beta$ being of rank one
 nor two. 
\label{prop:thno}
\end{proposition}

\begin{proof}
The argument consists of the following three steps:  
first, we consider the period map $v$ defined from 
$\mathcal{T}_{( c_{0} , p_{0})}$ 
to  $\text{Fred}_{n}\, \left( V_1 , V_2 \right) $. 
As the differential of $v$ is not necessarily surjective, we {\it
 stabilize} the map to obtain a submersion $v' : 
\mathcal{V}_{(c_0, p_0)} \times \mathcal{T}_{(c_{0},p_{0})} \to
 \text{Fred}_{n} (V_1 , V_2)$, where $\mathcal{V}_{(c_{0},p_{0})}$ is
 some finite dimensional vector space in $T_{v (c_0 , p_0)}
 \text{Fred}_{n} (V_1 ,V_2)$. 
Second, we lift the stabilized period map $v'$ to
	     $\mathcal{V}_{(c_0, p_0)} \times
 \mathcal{T}_{(c_{0},p_{0})} \to \text{Flag}_{k,n} (V_1 ,V_2)$ as 
the rank one and two loci $\tilde{J}_{k} (Z')$ lives in
 $\text{Flag}_{k,n} (V_1 , V_2)$.  
This is again not necessarily a submersion, so we stabilize it to obtain
 a smooth submersion 
$w' : \C^{k} \times \mathcal{W}_{(c_0, p_0),k} \to 
\text{Flag}_{k,n} (V_1, V_2)$, 
where $\mathcal{W}_{(c_{0},p_{0}),k}$ is a submanifold of 
$\mathcal{V}_{(c_0 ,p_0)} \times \mathcal{T}_{c_0, p_0}$ with finite
 codimension. 
Third, we use the Sard--Smale theorem (Proposition \ref{prop:SS}) 
to the $w'$ to obtain the assertion.

\vspace{0.2cm}

\noindent
{\it \underline{Step 1.} }
First, we consider the  period map 
$v 
: \mathcal{T}_{( c_{0} , p_{0})} 
\to \text{Fred}_{n}\, \left( V_1 , V_2 \right) $. 
From Proposition \ref{prop:psurj}, 
the operator 
$$ (D v)_{(c_0, p_{0})} 
: \{ 0 \} \oplus T_{p_{0}} \mathcal{P}_2 
\to T_{v (c_0,  p_{0})} 
\text{Fred}_{n} \, \left( 
V_1 ,V_2  \right) $$
is surjective. 
On the other hand, we have 
$$
T_{( c_0, p_{0})} 
\mathcal{T}_{(c_0  , p_{0})} 
+ 
\left( \{ 0 \} \oplus 
T_{p_0} \mathcal{P}_2 \right)
= \mathbb{H}^{1}_{( c_0 , p_{0})} \oplus 
T_{p_{0}} \mathcal{P}_2 , $$
where 
\begin{equation*}
\begin{split}
\mathbb{H}^{1}_{(c_0 , p_{0})} 
&:= \ker \left( D s (\cdot , p_{0} ) \right)_{(c_0, p_{0})} \\
& = \ker \left( \Pi_{(c_0, p_0)} \circ (D s) ( \cdot , p_{0})
 \right)_{(c_{0} , p_{0} )} \\
&= \ker \left( D \pi_{\mathcal{T}, \mathcal{P}_2} \right)_{(c_0 ,p_0)}
\subset T_{(c_0, p_{0})}
\mathcal{C}^{*,0} (P). 
\end{split}
\end{equation*}
Hence, 
$ (D v)_{(c_0, p_{0})} 
: \mathbb{H}^{1}_{( c_0 , p_{0})} \oplus T_{p_{0}} \mathcal{P}_2 
\to T_{v (c_0 ,p_0)} 
\text{Fred}_{n} \, \left( 
V_1 , V_2 \right) $ 
is surjective.

As \cite[Lem~4.15]{F}, 
we also have the following isomorphism. 
\begin{equation*}
\left( \mathbb{H}^{1}_{(c_0, p_{0})} 
 \oplus T_{p_0} \mathcal{P}_2\right) 
  \cong 
 T_{(c_0 , p_{0})} 
\mathcal{T}_{( c_0 , p_{0})} 
\oplus 
\text{\rm coker}\,  \left( D \pi_{ \mathcal{T} ,
		     \mathcal{P}_2}\right)_{( c_0, p_{0})}. 
\end{equation*}
We then define the following finite dimensional vector space. 
\begin{equation*}
V_{(c_0, p_{0})}
 := \left( D v \right)_{(c_0 ,
 p_{0})} 
\left( \text{coker}\, 
 (D \pi_{\mathcal{T} , \mathcal{P}_2})_{( c_0, p_{0})} \right) 
\subset T_{v ( c_0, p_{0})} 
 \text{Fred}_{n} \, \left( 
V_1 , V_2 \right) . 
\end{equation*}
We denote the inclusion by 
$i : 
V_{(c_0 , p_{0})} 
\to T_{v ( c_0 , p_{0})} 
\text{Fred}_{n} \, \left( V_1 ,V_2 \right)  $. 
We then define 
$$v' : 
V_{ (c_{0} , p_{0})} 
\times \mathcal{T}_{(c_{0}, p_{0})} 
\to 
\text{Fred}_{n} \, \left( V_1 , V_2  \right) $$
by 
$v' (y , (c,p)) 
:= i(y) + v (c,p)$ for 
$(y , (c,p)) \in V_{(c_0, p_{0})} 
\times \mathcal{T}_{( c_0 , p_{0})} 
$. 
As the differential of $v'$ is surjective at 
$(0, c_0, p_{0} )$, 
there exists an open neighbourhood of the origin 
$\mathcal{V}_{(c_0, p_{0})}
 \subset 
V_{(c_0 , p_{0})}$ such that the restriction 
\begin{equation}
v' :
\mathcal{V}_{(c_0 ,
p_{0})} \times \mathcal{T}_{(c_0, p_0)} \to 
\text{Fred}_{n} \, \left( V_1, V_2  \right) 
\label{eq:ppsub}
\end{equation}
is a submersion.

We now consider the following for $k \geq n$.  
\begin{equation*}
\mathcal{W}_{(c_0,
p_{0}), k} := (\mathcal{V}_{( c_0,
p_{0})} \times \mathcal{T}_{( c_0,
p_{0})} ) \cap 
(v')^{-1} ( \text{Fred}_{k,n}\, \left( 
V_1 , V_2 \right) ). 
\end{equation*} 
As \eqref{eq:ppsub} is a submersion, the above $\mathcal{W}_{( c_0,
p_{0}), k}$ is a smooth submanifold with finite codimension in 
$\mathcal{V}_{( c_0, p_{0})} \times \mathcal{T}_{( c_0 ,
p_{0})} $, thus, 
$\mathcal{V}_{(c_{0} , p_{0})} \times \mathcal{T}_{(c_0, p_0)} 
= \bigcup_{k \geq n} \mathcal{W}_{(c_{0}, p_{0}),k}$ is a countable
 disjoint union of smooth manifolds.

\vspace{0.2cm}

\noindent
{\it \underline{Step 2.} }
We next lift the map 
$v' : 
\mathcal{W}_{(c_0,
p_{0}), k}  \to 
\text{\rm Fred}_{k,n}  \left( 
V_1 , V_2  \right) $ 
to a smooth map 
$$w : 
\mathcal{W}_{(c_0, p_0),k} \to \text{Flag}_{k,n} \left( 
V_1 ,V_2\right) 
$$
by 
$(y , (A ,  (\alpha , \beta ) ), p ) 
 \mapsto 
( [ ( \alpha , \beta )] , i (y)  +  v ( (A , (\alpha , \beta )),
 p)$. 
This is again not necessarily a submersion, so we stabilize it 
as described below.

Let $(y_1 , (c_1 , p_1) )$ 
in $\mathcal{W}_{( c_0 ,p_{0}),k}$. 
Since a countable union of first category subsets is a first category
subset and $\mathcal{W}_{(c_0, p_0) ,k}$ is paracompact, 
we only consider a single open neighbourhood of $(y_1 , (c_1 , p_1))$.

We take an orthonormal basis $\{ b_{1, j} \}_{j=1}^{k}$ of the kernel of
$v' 
(y_1 , ( c_1 , p_1)) = 
i (y_1) + v(c_1 , p_1) $. 
We denote by 
$$ \pi_{(y , (c ,p))} 
:  L^{2}_{k-1} (\mathfrak{g}_{P} \otimes (\Lambda^{0,0} \oplus \Lambda^{0,2})) 
\to \ker \left( i (y) + v(c,p) \right) $$ 
the smooth family of $L^2$-orthogonal projection.  
We then consider a smooth map 
$$w' : 
\C^{k} \times 
\mathcal{W}_{(c_0 ,p_{0}),k} 
\to 
\text{Flag}_{k,n} 
 (  V_1 , V_2 )$$
defined by 
\begin{equation*}
w' 
(z , y , c , p) 
\mapsto 
 ( [ (\alpha , \beta) + \pi_{(y, (c , p))} 
(\sum_{j=1}^{k} z_{j} b_{1,j})], i(y) + v(c,p) ), 
\end{equation*}
where 
$z = ( z_{1} , \dots , z_{k}) \in \C^{k}$. 
As \cite[Claim~4.18]{F},  
the map 
$w'$ is a submersion at 
$( 0, y_1 , (c_1 , p_1))$, 
thus  
$ \mathcal{W}_{( c_0 ,p_{0}),k}' 
:= 
( w')^{-1} \left( \tilde{J}_{k}
(Z')  \right)$ 
is a $C^{\infty}$-Banach submanifold of 
$\C^{k} \times \mathcal{W}_{( c_0 ,p_{0}),k}$.

\vspace{0.2cm}

\noindent
{\it \underline{Step 3.} }
We are now in a situation to invoke the Sard--Smale theorem (Prop. \ref{prop:SS}). 
Applying it to $w'$, we obtain a first
category subset $\mathcal{P}_2' \subset \mathcal{P}_2$ such that 
for $p \in \mathcal{P}_2 \setminus \mathcal{P}_2'$ 
\begin{equation*}
 \text{\rm codim}_{\R} 
 \left( \mathcal{W}_{( c_0 ,p_{0}),k}' |_{p} ,
\C^{k} \times 
\mathcal{W}_{ 
 (c_0, p_{0}),k}|_{p}  
 \right)  
 = \text{\rm codim}_{\R}  \left( \tilde{J}_{k} (\mathcal{\mathcal{Z}}') , 
\text{\rm Flag}_{k,n} ( V_1 , V_2) \right) . 
\end{equation*} 
Since $\text{\rm codim}_{\R}  \left( \tilde{J}_{k} (\mathcal{\mathcal{Z}}') , 
\text{\rm Flag}_{k,n} ( V_1 , V_2) \right) = \infty$ but  
$\dim_{\R} 
\left( \C^{k} \times 
\mathcal{W}_{( c_0 , p_{0}),k}|_{p}  \right) 
< \infty $, 
we deduce that $\mathcal{W}_{( c_0  ,p_{0}),k}' |_{p}$ is empty.

We also have 
$\mathcal{T}_{( c_0 ,p_{0})}|_{p} \cap 
w|_{\mathcal{T}_{(c_0 , p_0 )}} ( \cdot , p)^{-1} 
( \tilde{J}_{k} (\mathcal{Z}')) \subset \mathcal{W}_{( c_0 ,p_{0}) ,k}'|_{p}$. 
Since $\mathcal{W}_{( c_0 ,p_{0}),k}' |_{p}$ is empty,  
thus so is $\mathcal{T}_{( c_0 ,p_{0})}|_{p} \cap 
w|_{\mathcal{T}_{(c_0 , p_0 )}} ( \cdot , p)^{-1} 
( \tilde{J}_{k} (\mathcal{Z}'))$.  
Hence $\mathcal{T}_{(c_0, p_0)}$ has no rank one or two section $\alpha
 + \beta $ for $\dim \ker (\bar{\partial}_{A,p} + \bar{\partial}^{*}_{A,p}
 ) =k$ and $p \in \mathcal{P}_2 \setminus \mathcal{P}_2'$. 
Since a countable union of first category subsets is a first category
 subset, we get the assertion by repeating this for $k \geq n$. 
\end{proof}

\vspace{0.2cm}

\noindent
{\it Proof of Proposition \ref{th:norot}.} 
By Proposition \ref{prop:thno}, $M^{*,0} (P) \cap \mathcal{T}_{(c_0 ,p_{0})} 
\subset \mathcal{T}_{( c_0 ,p_{0})}$ has no rank one nor two solution 
$(A, (\alpha , \beta) , p)$ for $p \in \mathcal{P}_2 \setminus \mathcal{P}_2'$. 
By repeating this argument for each $(A , (\alpha ,\beta) , p) \in 
\mathcal{C}^{*,0} (P) \times \mathcal{P}_2$, we obtain a first category subset 
for each open neighbourhood of it. 
As $\mathcal{C}^{*,0} (P) \times \mathcal{P}_2$ is paracompact, we can
cover $M^{*,0} (P)$ by countable such open neighbourhoods. 
Since a countable union of first category subsets of $\mathcal{P}_2$ is
again 
a first category subset of $\mathcal{P}_2$, we get the assertion. 
\qed


\begin{flushleft}
Mathematical Institute, University of Oxford \\
Radcliffe Observatory Quarter, Woodstock Road, Oxford, OX2 6GG, U.K.\\
tanaka@maths.ox.ac.uk
\end{flushleft}


\end{document}